\documentclass[12pt]{article}
\usepackage{latexsym}
\usepackage[reqno, namelimits, sumlimits]{amsmath} 
\usepackage{amssymb, amsthm, amsfonts}
\usepackage{amsthm} 
\usepackage[english]{babel}
\usepackage{enumitem}

\newtheorem{Theorem}{Theorem}[section]
\newtheorem{Lemma}[Theorem]{Lemma}

\newtheorem{theorem}[Theorem]{Theorem}
\newtheorem{lemma}[Theorem]{Lemma}
\newtheorem{proposition}[Theorem]{Proposition}

\newtheorem{remark}[Theorem]{Remark}

\newcommand{\R}{{\mathbb R}}

\newcommand{\cS}{{\mathcal S}}

\newcommand{\al}{\alpha}
\newcommand{\be}{\beta}
\newcommand{\ga}{\gamma}

\newcommand{\Ga}{\Gamma}
\newcommand{\de}{\delta}
\newcommand{\ep}{\epsilon}
\newcommand{\ze}{\zeta}
\newcommand{\De}{\Delta}

\newcommand{\om}{\omega}

\newcommand{\la}{\lambda}

\newcommand{\varep}{\varepsilon}

\newcommand{\codt}{\cdot}

\newcommand{\superset}{\supset}

\newcommand{\rest}{\big\arrowvert}
\newcommand{\upto}{\nearrow}
\newcommand{\downto}{\searrow}

\renewcommand{\superset}{\supset}

\numberwithin{equation}{section}

 \title{On the multiplicity of self-similar solutions of the semilinear heat equation}

 \author{P. Pol\'a\v cik\footnote{Supported in part by NSF Grant DMS-1565388 }
\\
\small School of Mathematics, University of Minnesota\\
\small Minneapolis, MN 55455
\\~\\
 P. Quittner\footnote{Supported in part by VEGA Grant 1/0347/18
 and by the Slovak Research and Development Agency under the contract No. APVV-14-0378} 
\\
\small Department of Applied Mathematics and Statistics, 
       Comenius University,\\
\small Mlynsk\'a dolina, 84248 Bratislava, Slovakia}

\date{} 

\begin{document}

\maketitle

\begin{quote}{\small
    {\bfseries Abstract.}
    In studies of superlinear parabolic equations
\begin{equation*}
  u_t=\Delta u+u^p,\quad x\in {\mathbb R}^N,\ t>0,
\end{equation*}
where $p>1$, backward self-similar solutions play an important role.
These are solutions of the  form $ u(x,t) = (T-t)^{-1/(p-1)}w(y)$, where
$y:=x/\sqrt{T-t}$, $T$ is a constant, and $w$ is
a solution of the equation
$\Delta w-y\cdot\nabla w/2 -w/(p-1)+w^p=0$. We consider (classical)
positive radial solutions $w$ of this equation.
Denoting by $p_S$, $p_{JL}$, $p_L$ the Sobolev, Joseph-Lundgren, and
Lepin exponents, respectively, 
we show that for $p\in (p_S,p_{JL})$ there are only
countably many solutions, and for  
$p\in (p_{JL},p_L)$ there are only finitely many solutions.
This result answers two basic open questions
regarding the multiplicity of the solutions. 
} 
 \end{quote}

{\emph{Key words}: Semilinear heat equation, self-similar solutions,
  multiplicity of solutions, shooting techniques, Laguerre polynomials} 

{\emph{AMS Classification:} 35K57, 35C06, 35B44,  35J61}

\tableofcontents

\section{Introduction}\label{intro}
In this paper, we consider positive radial self-similar solutions of the
semilinear heat equation
\begin{equation}
  \label{eq:sh}
  u_t=\De u+u^p,\quad x\in \R^N,\ t>0,
\end{equation}
where $p>1$ is a real number. A radial (backward) self-similar solution $u$ is a
solution of the form $ u(x,t) = (T-t)^{-1/(p-1)}w(|y|)$, where
$y:=x/\sqrt{T-t}$, $T$ is a constant, and $w$ is a function
in $C^1[0,\infty)$. Such a function $u$ is a (regular) positive  solution
of \eqref{eq:sh} if  $w$ is a solution of the following problem
\begin{align} \label{eq:w}
w_{rr}+\left(\frac{N-1}r-\frac r2\right) w_r-\frac w{p-1}+w^p&=0,\quad
  r>0,\\
  w_r(0)=0,\quad w&>0. \label{eq:cond}
\end{align}

Self-similar solutions have an indispensable role in the theory of
blowup of equation \eqref{eq:sh}. They are examples of solutions
exhibiting  \emph{type-I} 
blowup at time $T$, by which we mean that the rate
of blowup is $(T-t)^{-1/(p-1)}$, the same as in the ordinary
differential equation $u_t=u^p$. In fact, 
they often serve as canonical examples in the sense   
that general solutions with type-I blowup can be proved to
approach in some way a self-similar
solution  as $t$ approaches the blowup time
(see, for example,
\cite{Bebernes-E:char,Giga-K:asym,Giga-K:char,Matos:conv,Matos:self-sim},
or the monograph \cite{Quittner-S:bk} for results of this form). 
Moreover, self-similar solutions play an important role  
in the study of the asymptotic behavior of global solutions
(see \cite{Galaktionov-V:cont} or
\cite[the proof of Theorem 22.4]{Quittner-S:bk}, for example),
in the construction of interesting solutions
(like peaking or homoclinic solutions;
see \cite{Galaktionov-V:cont} or \cite{FY11}, respectively),
in the study of type-II blow-up (see \cite[Proposition 1.8(ii)]{Matano-M:11}), 
etc. 

Equation \eqref{eq:w} has been scrutinized by a number of
authors. To recall known results, we introduce several critical
exponents (they are usually called the Sobolev, Joseph-Lundgren, and
Lepin exponents, respectively):
\begin{align*}
 p_S&:=\begin{cases}
     \frac{N+2}{N-2} &\hbox{ if }N>2, \\
     \infty &\hbox{ if }N\leq2,
   \end{cases}\\
   p_{JL}&:=\begin{cases}
     1+4\frac{N-4+2\sqrt{N-1}}{(N-2)(N-10)} 
        &\hbox{ if }N>10, \\
     \infty &\hbox{ if }N\leq10,
   \end{cases}\\
p_L&:=\begin{cases}
     1+\frac6{N-10}  &\hbox{ if }N>10, \\
     \infty &\hbox{ if }N\leq10.
   \end{cases} 
 \end{align*}
   
 Obviously, the constant $\kappa:=(p-1)^{-1/(p-1)}$ is a solution
 of \eqref{eq:w} for any $p>1$. If $1<p\le p_S$, $\kappa$ is the only
 positive solution  
 \cite{Giga-K:asym,Giga-K:nondeg}.
 For $p_S<p<p_{JL}$, there exist at least countably 
 (infinitely) many  solutions; and for $p_{JL}\le p<p_L$ the existence
 of a (positive) finite number of nonconstant solutions has been
 established
(see \cite{Budd-Qi, Collot-R-S, Fila-P, Galaktionov-V:cont, Lepin, Troy}).
For $p> p_L$, $\kappa$ is again the only positive solution. This was
proved by \cite{Mizoguchi:nonexitence} (the
nonexistence was indicated by  numerical
experiments in the preceding paper \cite{Sverak-P}).
The same seems to be the case for $p=p_L$, as claimed in 
\cite{Mizoguchi:backward}, however the proof given in
\cite{Mizoguchi:backward} is not complete.

A natural and rather basic question, which, despite its importance on
several levels, has been open until
now  is whether there may exist infinitely many solutions for some
$p\in [p_{JL},p_L)$, or uncountably many solutions for some
$p\in (p_S,p_{JL})$. In particular, it is of significance to clarify
whether there might  be continua of solutions
\eqref{eq:w}, \eqref{eq:cond} for some $p\in (p_S,p_{L})$.
For example, by ruling out the possibility that such continua exist,
one could substantially simplify the proofs of 
some results on self-similar asymptotics of blowup solutions of
\eqref{eq:sh}, such as those in 
\cite{Matos:conv} or \cite[Theorem 3.1]{Matano-M:09}.
Also, the problems of finiteness and countability of the set of the
radial self-similar solutions are of great
importance in our study of  entire and ancient solutions of
\eqref{eq:sh}, which will appear in a forthcoming paper
\cite{p-Q:entire}.   

The goal of the present work is to address these problems.
Our main results are stated in the following theorem. 
\begin{theorem}
  \label{thm:main}
  Under the above notation, the following statements are valid.
  \begin{itemize}
  \item[(i)] For any  $p\in (p_S,p_{JL})$ the set of solutions of
    \eqref{eq:w}, \eqref{eq:cond} is infinite and countable. For $p=p_{JL}$
    the set of solutions of 
    \eqref{eq:w}, \eqref{eq:cond} is at most countable. 
  \item[(ii)] For any $p\in (p_{JL},p_L)$ the set of solutions of
    \eqref{eq:w}, \eqref{eq:cond} is finite. 
  \end{itemize}
\end{theorem}

Very likely, in the case $p=p_{JL}$ the set of solutions is finite, too,
but our proof of the finiteness for $p\in (p_{JL},p_L)$ does not
cover the case $p=p_{JL}$.    
On the other hand, if there exist solutions for $p=p_L$,  
then our proof guarantees that the set of solutions is finite. 
 
Our method relies on two kinds of shooting techniques; one from the
origin, considering a standard initial-value problem for \eqref{eq:w}
at $r=0$, and another one where ``initial conditions'' are prescribed
at $r=\infty$. In the proof of statement (i), we
employ the analyticity of the nonlinearity
$u\mapsto u^p$ in $(0,\infty)$. We prove that
the solutions are isolated, hence
there is at most countably many of them.
Technical difficulties in this proof are caused
by the fact that as $r\to\infty$ the nonconstant
solutions of   \eqref{eq:w}, \eqref{eq:cond} decay to 0,
where we loose the analyticity.   
For the proof of statement (ii), we
show that the solutions cannot accumulate at the
singular solution. This involves a subtle analysis of how solutions of
the initial value problems at $r=0$ and $r=\infty$ behave near the
singular solution.

The paper is organized as follows. In the next section, we introduce
some notation and recall several technical results concerning solutions
of \eqref{eq:w}. In Section \ref{proof1}, we use shooting arguments
to show that the solutions of \eqref{eq:w} are in one-to-one
correspondence with the zeros of a real analytic function. This is
key to showing that the solutions are isolated. In
Section \ref{proof2}, we consider the set of solutions of \eqref{eq:w} 
near a  singular solution and complete 
the proof of  Theorem \ref{thm:main}(ii).

\section{Notation and preliminaries}
\label{prelim}
In the remainder of the paper, it is always assumed that $p>p_S$.

Although equation \eqref{eq:w} has a singularity at $r=0$,
it is well-known and easy to prove by an application
of the Banach fixed point theorem to an integral operator
(see, for example, \cite{Haraux-W}) that for
each $\al>0$ there is a unique local solution of \eqref{eq:w}  
satisfying the initial conditions
\begin{equation}
  \label{eq:ic}
   w_r(0)=0,\ w(0)=\al. 
\end{equation}
We denote this solution by $w(r,\al)$ and extend it to its maximal
existence interval. If the solution changes sign, then the nonlinearity
in \eqref{eq:w} is interpreted as $w|w|^{p-1}$. Let
\begin{equation}
  \label{eq:1}
  \cS:=\{\al>0: w(r,\al)>0\ \,(r\in (0,\infty))\}.
\end{equation}
Obviously, for each solution $w$ of \eqref{eq:w} one has
$w=w(\cdot,\al)$ for some (unique) $\al\in \cS$.
We further denote
\begin{equation}
  \label{eq:3}
  \phi_\infty(x):=L|x|^{-2/(p-1)}, \quad
L:=\left(\frac2{(p-1)^2}((N-2)p-N)\right)^{\frac1{p-1}}.
\end{equation}
$$ $$
This is a singular solution of \eqref{eq:w} (it is defined
when $p(N-2)>N$). In fact, this is a unique solution of
\eqref{eq:w} with a singularity at $r=0$ (see
\cite{Mizoguchi:backward, Quittner:uniqueness}).

We recall the following properties (as above, $\kappa=(p-1)^{-1/(p-1)}$):
\begin{Lemma}
  \label{le:basic}
 The following statements are valid (for each $p>p_S$).  
 \begin{itemize}
 \item[(i)] One has $\al\ge \kappa$ for all $\al\in \cS$,
   and $\kappa$ is isolated in $\cS$.
 \item[(ii)] For each $\al\in \cS\setminus\{\kappa\}$ there exist a positive
   constant $\ell(\al)$ such that
   \begin{equation}
     \label{eq:2}
     w(r,\al)=\ell(\al)r^{-\frac{2}{p-1}}(1-c(\al)r^{-2}+o(r^{-2}))\quad\text{as
       $r\to\infty$}, 
   \end{equation}
   where $c(\al):=(\ell(\al))^{p-1}-L^{p-1}$ (with $L$ as in
   \eqref{eq:3}). Moreover, one has
   \begin{equation}
     \label{eq:4}
     \frac{w_r(r,\al)}{w(r,\al)}=-\frac{2}{(p-1)}r^{-1}+2c(\al)r^{-3}+o(r^{-3}) 
     \quad\text{as  $r\to\infty$}.
   \end{equation}
   \item[(iii)] With $\ell(\al)$ as in statement (ii),  function
     $\al\mapsto\ell(\al)$ is one-to-one on $\cS\setminus\{\kappa\}$.  
 \end{itemize}
\end{Lemma}
\begin{proof}
  Statement (i) is proved in  Lemmas 2.2 and 2.3 of
  \cite{Mizoguchi:nonexitence}. Statements (ii), (iii) for regular
  \emph{ bounded} solutions  are proved in \cite[Section~2]{Matos:conv};
  the boundedness assumption can be removed due to 
  \cite[Lemma 2.1]{Mizoguchi:nonexitence}.  
\end{proof}
We need two additional properties of the function $\al\mapsto
\ell(\al)$:
\begin{lemma}
  \label{le:cont}
 Suppose $\al_k\in\cS$, $k=1,2,\dots$, and $\al_k\to\al_0\in (\kappa,\infty]$ as
 $k\to\infty$. The following statements are valid.
 \begin{itemize}
 \item[i)] If $\al_0<\infty$, then $\al_0\in\cS$ and $\ell(\al_k)\to
   \ell(\al_0)$.
 \item[ii)] If $\al_0=\infty$ and $p>p_{JL}$,
   then $\ell(\al_n)\to
   L$.   
 \end{itemize}
\end{lemma}
\begin{proof}
  Statement (ii) is proved in \cite[Lemma
  2.7]{Mizoguchi:nonexitence}.

  We prove statement (i). It is clearly sufficient to prove that the
statement is valid if the sequence $\{\al_k\}_k$ is replaced by a
subsequence (and then use this conclusion for any subsequence of
$\{\al_k\}_k$ in place of the full sequence $\{\al_k\}_k$).
We may in particular assume that $\al_j\ne \al_k$ if $j\ne k$
(for a constant sequence the statement is trivially true).  

The fact that $\al_0\in\cS$, that is, the solution $w(\cdot,\al_0)$ is
positive, follows easily from the continuity of solutions with respect
to initial data. 

Consider now the function
$v(r,\al):=w(r,\al)r^{2/(p-1)}$. It solves the equation
\begin{equation} \label{eq:v}
v_{rr}+\Bigl(\frac{N-1-4/(p-1)}r-\frac r2\Bigr)v_r+\frac1{r^2}(v^p-{L}^{p-1}v)=0,
\quad r>0.  
\end{equation}
If $\al,\bar\al\in{\cS}$, $\al\ne\bar\al$ and 
$h(r):=v(r,\al)-v(r,\bar\al)$, then $h$ solves the equation
\begin{equation} \label{eq:h}
h_{rr}+\Bigl(\frac{N-1-4/(p-1)}r-\frac r2\Bigr)h_r
  +\frac1{r^2}(pv_\theta^{p-1}-{L}^{p-1})h=0,
\quad r>0,
\end{equation}
where $v_\theta=v_\theta(r)$ belongs to the interval with end points 
$v(r,\al)$ and $v(r,\bar\al)$.

It can be shown that if $R>0$ is sufficiently large, then 
for any $k\ne j$ and $r>R$ one has
$v(r,\al_k)\ne v(r,\al_j)$  (for example, in the proof
of Proposition 2.4 in \cite{Fila-P} it was shown that one can take 
any $R>\sqrt{2N}$).
Fix any $R$ with this property; 
from now on we consider the solutions $v$ on the interval
$[R,\infty)$ only. Passing to a subsequence,
we may assume that the sequence $\{v(R,\alpha_k)\}_{k\ge1}$ is
strictly monotone. Assume that it is decreasing (the other
case is analogous),
hence $v_k:=v(\cdot,\alpha_k)$ satisfy
$v_1>v_2>\dots>v_0$ on $[R,\infty)$, $v_k\to v_0$ in $C_{loc}([R,\infty))$.

Set $h_k:=v_{k}-v_{k+1}$, $k=1,2,\dots$.
Then $h_k$ is positive and it solves \eqref{eq:h}
with $v_{k+1}<v_\theta<v_k$. Moreover,
$h_k\to0$ in $C_{loc}([R,\infty))$
and $\sum_k h_k\leq v_1-v_0\leq C$, for some constant $C>0$. 

First assume 
\begin{equation} \label{v0-large}
pv_0^{p-1}>L^{p-1} \quad\hbox{on }\ [R,\infty).
\end{equation}
If $h_k'(r_0)\leq0$ for some $r_0\geq R$ then \eqref{eq:h}
guarantees $h_k''(r_0)<0$, hence $h_k',h_k''<0$ for $r>r_0$
which contradicts the positivity of $h$.
Consequently, $h_k'\geq0$.

Set $h_k^\infty:=\lim_{r\to\infty} h_k(r)$.
Since 
\begin{equation} \label{v1v0}
v_1(r)-v_0(r)=\sum_{k=1}^\infty h_k(r)\nearrow\sum_{k=1}^\infty h_k^\infty
\quad\hbox{as}\quad r\to\infty,
\end{equation}
we have
$\sum_{k=1}^\infty h_k^\infty<\infty$. Fix $\varep>0$.
Then there exists $k_0$ such that 
$\sum_{k=k_0}^\infty h_k^\infty<\varep$,
hence $v_{j}(r)-v_0(r)=\sum_{k=j}^\infty h_k(r)<\varep$ for any $r\in[R,\infty)$
and $j\geq k_0$.
This implies
$\ell(\alpha_0)<\ell(\alpha_{j})\leq\ell(\alpha_0)+\varep$,
hence $\ell(\alpha_k)\to\ell(\alpha_0)$.

Next assume that \eqref{v0-large} fails.
Notice that there exist $c_1,c_2>0$ such that
$$ \frac{N-1-4/(p-1)}r-\frac r2\leq -c_1r \quad\hbox{and}\quad
  pv_\theta^{p-1}-L^{p-1}\geq-c_2 \quad\hbox{for }\ r\geq R.$$
Set $c_3:=c_2/c_1$.
If $h_k'(r_0)<-c_3h_k(r_0)r_0^{-3}$ for some $r_0\geq R$ 
then \eqref{eq:h} guarantees $h_k''(r_0)<0$,
hence $h_k'(r)<-c_3h_k(r)r^{-3}$ 
and $h_k''(r)<0$ for $r>r_0$
(since $r\mapsto -c_3h_k(r)r^{-3}$ is increasing if $h_k'<0$), 
which contradicts the positivity of $h$.
Consequently, $h_k'\geq -c_3h_k(r)r^{-3}$, or, equivalently,
$(e^{-c_3/(2r^2)}h_k)'\geq0$.

Set $h_k^\infty:=\lim_{r\to\infty} h_k(r)
  =\lim_{r\to\infty}e^{-c_3/(2r^2)}h_k(r)$.  
Now
$$e^{-c_3/(2r^2)}(v_1(r)-v_0(r))=
 \sum_{k=1}^\infty e^{-c_3/(2r^2)}h_k(r)
 \nearrow\sum_{k=1}^\infty h_k^\infty\quad\hbox{as}\quad r\to\infty,$$
and similar arguments as above show  that
$\ell(\alpha_k)\to\ell(\alpha_0)$ again.
\end{proof}

We will also need the following information on the behavior of
the solutions $w(\codt,\al)$ for large $\al$. This result---in fact, a
stronger version of it---is proved in
\cite[Lemma 2.5]{Mizoguchi:nonexitence}.
\begin{lemma}
  \label{le:large}
  Assume that $p>p_{JL}$. Then, as $\al\upto\infty$, one has
  \begin{equation}
    \label{eq:5}
    w(r,\al)\to\phi_\infty(r),\quad w_r(r,\al)\to\phi'_\infty(r),
  \end{equation}
  uniformly for $r$ in any compact subinterval of 
  $(0,\infty)$. 
\end{lemma}

In some comparison arguments below, we will employ radial
eigenfunctions of the linearization of \eqref{eq:w}
at the singular solution $\phi_\infty$. Specifically,
we consider the following eigenvalue problem:
\begin{equation}
  \label{eq:17}
  \begin{aligned}
    \psi_{rr}+\left(\frac{N-1}r-\frac r2\right) \psi_r+\left(-\frac
      1{p-1}+\frac{pL^{p-1}}{r^2}+\la\right)\psi=0&,\quad r>0\\
    \psi\in H^1_\om(0,\infty)&.
  \end{aligned}
\end{equation}
Here  $H^1_\om(0,\infty)$ is the usual weighted Sobolev space with the
weight 
\begin{equation} \label{eq:anew}  
\om(r):=r^{N-1}\exp(-r^2/4).
\end{equation} 
The inclusion
$\psi\in H^1_\om(0,\infty)$ means that if $\tilde \psi$ equals $\psi$
or $\psi'$, 
then
\begin{equation*}
  \int_0^\infty \tilde \psi^2(r)\om(r)\,dr<\infty.
\end{equation*}
This eigenvalue problem is well understood.
The following lemma summarizes some 
basic known results
(see \cite{Herrero-V1,Mizoguchi:blowup}).
\begin{lemma}\label{le:HV}  Assume that  $p>p_{JL}$.
  The eigenvalues of \eqref{eq:17} form a sequence explicitly given by
  \begin{equation}
    \label{eq:18}
    \la_j=\frac{\be}2+\frac{1}{p-1}+j,\quad j=0,1,2, \dots,
  \end{equation}
  where
  \begin{equation}
    \label{eq:19}
    \be:=\frac{-(N-2)+\sqrt{(N-2)^2-4pL^{p-1}}}{2}<0.
  \end{equation}
  For $j=0,1,2,\dots$, the eigenfunction corresponding to $\la_j$,
  which is unique up to scalar multiples, has exactly $j$ zeros, all
  of them simple, and satisfies the following asymptotic relations
  with some positive constants $k_j$, $\tilde k_j$:
  \begin{align*}
    \psi_j(r)&=k_jr^{\be}+o(r^{\be})\quad\text{as $r\to0$,}\\
    \psi_j(r)&=\tilde
               k_jr^{-\frac2{p-1}+2\la_j}+o(r^{-\frac2{p-1}+2\la_j})
               \quad\text{as $r\to\infty$.}
  \end{align*}
  If $p_{JL}<p<p_L$, then $\la_2<0$ and if $p=p_L$, then
  $\la_2=0$.   
\end{lemma}

The following result, which is a part of analysis used
in \cite{Herrero-V1,Mizoguchi:blowup},
will also be useful below. It can be easily derived from
well-known properties of Kummer's equation 
(as shown in \cite{Herrero-V1,Mizoguchi:blowup}).
Consider the following equation (the same equation as in
\eqref{eq:17}, but with $\la=0$).
\begin{equation}
  \label{eq:170}
    \psi_{rr}+\left(\frac{N-1}r-\frac r2\right) \psi_r+\left(-\frac
      1{p-1}+\frac{pL^{p-1}}{r^2}\right)\psi=0,\quad r>0
  \end{equation}

\begin{lemma}
  \label{le:anal}  Assume that $p>p_{JL}$. Equation \eqref{eq:170} has
  (linearly independent)
  solutions $\psi_1$, $\psi_2$ satisfying the following 
  asymptotic relations with some positive constants
  $\kappa_1$, $\kappa_2$: 
  \begin{align}\label{eq:30a}
    \psi_1(r)&=\kappa_1  r^{\be}+o(r^{\be})\quad\text{as $r\to0$,}\\
    \psi_2(r)&=\kappa_2  r^{\be^-}+o(r^{\be})\quad\text{as $r\to0$.}
             \label{eq:30b}
  \end{align}
  Here $\be$ is as in \eqref{eq:19} and
 \begin{equation}
    \label{eq:19b}
    \be^-:=\frac{-(N-2)-\sqrt{(N-2)^2-4pL^{p-1}}}{2}<\be<0.
  \end{equation}
  Problem \eqref{eq:17} has $\la=0$ as an eigenvalue precisely when
  $\psi_1$ also satisfies the following 
  asymptotic relation with some positive constant
  $\tilde \kappa_1$
  \begin{equation}
    \label{eq:31}
      \psi_1(r)= \tilde \kappa_1
          r^{-\frac2{p-1}}+o(r^{-\frac2{p-1}})
             \quad\text{as $r\to\infty$}.
  \end{equation}
\end{lemma}
Obviously, if \eqref{eq:31} holds, then
 $\psi_1$ is an eigenfunction corresponding to the
 eigenvalue $\la=0$.
 We remark that  $\psi_2$ cannot be an eigenfunction of \eqref{eq:17},
for \eqref{eq:30b}, \eqref{eq:19b} imply that it is not in
$H^1_\om$. 

We conclude this section  with a monotonicity property of
the function $\al\mapsto w(\codt,\al)$. It will be useful to note that
on any interval 
$(0,r_0]$ where $w(\codt,\al)>0$, 
$w_\al(\codt,\al)$ satisfies the linear equation
 \begin{equation}
  \label{eq:lin1}
  z_{rr}+\left(\frac{N-1}r-\frac r2\right) z_r+\left(-\frac
    1{p-1}+p(w(r,\al))^{p-1}\right)z=0,
\end{equation}
and it also  satisfies the initial conditions $w_\al(0,\al)=1$,
$w_{\al r}(0,\al)=0$. 

\begin{lemma}
  \label {le:pos}
  Assume $p>p_{JL}$. There exist positive
  constants $\al^*$, $R$, and $C_1$ such that for all $\al>\al^*$ and
  $r\in [0,R]$    one has
  \begin{align}
    \label{eq:241}
    w_\al(r,\al)&>0\quad (r\in [0,R])\\
    \frac{w_\al(r,\al)}{ w_\al(r_0,\al)}&\le C_1r^{\be}\quad (r\in
      [0,R],\  r_0\in [R/2,R]), \label{eq:242}
 \end{align}
 where $\be$ is as in \eqref{eq:19}.
\end{lemma}

\begin{proof}
  All arguments needed for the proof of these estimates are
  essentially
  given in the proof
  of Lemma 2.8 of \cite{Mizoguchi:nonexitence},
  although the estimates  are not formulated
  there explicitly for the same functions.  We
  give a brief sketch of the proof.
 
  Take an eigenfunction $\psi_j$ of \eqref{eq:17} corresponding to a
  positive eigenvalue $\la_j$ and  let $r_1>0$ be its first zero.
  We will assume that $\psi_j>0$ in $(0,r_1)$
  (replace $\psi_j$ by $-\psi_j$ if
  necessary). Set
  $R=r_1/2$. 
  Considering the linear equation for   
  $\phi_\infty-w(\codt,\al)$ and using a Sturmian comparison with
  $\psi_j$, it is shown in the proof of Lemma 2.8 of
  \cite{Mizoguchi:nonexitence} that if $\al$ is large
  enough, then the first
  zero of   $\phi_\infty-w(\codt,\al)$ is greater  than $r_1$,
  that is,
  \begin{equation}
    \label{eq:27}
    w(r,\al)<\phi_\infty(r)\quad (r\in [0,2R]). 
  \end{equation}
  Now, considering the linear equation \eqref{eq:lin1}
  satisfied by $w_\al(\cdot,\al)$ 
and using the same kind 
  of Sturmian comparison with $\psi_j$, one shows that $w_\al(\cdot,\al)\ne 0$
  in $[0,r_1]$.   Since $w_\al(0,\al)=1$, we have
  $w_\al(\cdot,\al)>0$ 
  in $[0,R]$, proving \eqref{eq:241}. 

  An inequality similar to in \eqref{eq:242} is
  also proved
  in \cite[Proof of Lemma 2.8]{Mizoguchi:nonexitence}, but for a different
  function in place of $\tilde\psi(r,\al):={w_\al(r,\al)}/{
    w_\al(r_0,\al)}$. 
  However,  the same arguments apply to the function
  $\tilde\psi(r,\al)$ upon noting that for all large enough $\al$ the
  function 
  $\tilde\psi(r,\al)$ has the following two properties:
  \begin{itemize}
  \item[(1)] Equation \eqref{eq:lin1} has the zero order coefficient
  which is in  $[0,2R]$ smaller  than the zero order coefficient in 
  the equation for $\psi_j$, see \eqref{eq:17}.
  \item[(2)] One has
  $\tilde\psi(\cdot,\al)<C_2\psi_j$ on $[R/2,R]$, where  $C_2$ is a
  constant independent of $\al$. 
  \end{itemize}
  Property (1) follows from \eqref{eq:27}. To prove (2), we use the
  Harnack inequality for $\tilde\psi(r,\al)$---a positive solution of
  \eqref{eq:lin1}. Note that the coefficients of  \eqref{eq:lin1} are
  bounded in $[R/4,2R]$ uniformly in $\al$. This follows from
  \eqref{eq:27}. Since $\tilde\psi(r_0,\al)=1$, the Harnack inequality
  yields a uniform upper bound on $\tilde\psi(\cdot,\al)$ in $[R/2,R]$.
  Property (2) follows from this and the positivity of $\psi_j$
  on the interval  $(0,r_1]\superset [R/2,R]$. 
\end{proof}
\section{Shooting techniques and the proof of 
  Theorem \ref{thm:main}(i)}  
\label{proof1}
In this section, we employ two kinds of shooting arguments. The first one
is a standard shooting technique for \eqref{eq:w}, \eqref{eq:ic}
(shooting from $r=0$). The
second one is a kind of shooting from $r=\infty$, which becomes a more
standard shooting technique after equation \eqref{eq:w} is transformed
suitably. Shooting arguments of such sort were already used in
\cite{Lepin}, cf. also \cite{NS19}.  

\subsection{Shooting from $r=0$}
\label{shoot1}
We return to the initial-value problem \eqref{eq:w}, \eqref{eq:ic}.
As noted above, a local solution 
can be found in a standard way by applying the Banach fixed point 
theorem to a suitable integral operator. 
Since the nonlinearity $w\mapsto w^p$ is analytic
in intervals not containing $0$, the local solution depends
analytically on $\al$. Away from $r=0$, there are no singularities and
standard theory of ordinary differential equations applies. 
We thus obtain the following regularity property of the
function $w(r,\al)$.

\begin{lemma}
  \label{le:shoot1}
  Given any  $\al_0\in \cS$
  and $r_0\in (0,\infty)$, there is $\ep>0$
  with the following property.
  The solution $w(\cdot,\al)$ is (defined and) positive on
  $[0,2r_0]$ for any $\al\in (\al_0-\ep,\al_0+\ep)$,
  and the function  $w$ is analytic on
 $(0,2r_0) \times(\al_0-\ep,\al_0+\ep)$. 
\end{lemma}
Clearly, if $\al_0\in \cS$, then the
function  $w_\al(\codt,\al_0)$ solves on $(0,\infty)$
the linear equation
\begin{equation}
  \label{eq:lin}
  z_{rr}+\left(\frac{N-1}r-\frac r2\right) z_r+\left(-\frac
    1{p-1}+p(w(r,\al_0))^{p-1}\right)z=0, 
\end{equation}
and satisfies the initial conditions $w_\al(0,\al_0)=1$,
$w_{\al r}(0,\al_0)=0$. In particular, $w_\al(\codt,\al_0)$ is a
nontrivial solution of \eqref{eq:lin} and as such it has only simple
zeros.  

\subsection{Shooting from $r=\infty$}
\label{shoot2}
By Lemma \ref{le:basic}, if $u=w(\codt,\al)$ for some $\al\in
\cS\setminus\{\kappa\}$, then, as
       $r\to\infty$, one has
\begin{equation}
  \label{eq:6}
  u(r)=\ell r^{-\frac{2}{p-1}}(1+o(r^{-1})), \qquad
\frac{u'(r)}{u(r)}=-\frac{2}{(p-1)}r^{-1}+o(r^{-2}),
\end{equation}
for a suitable constant $\ell=\ell(\al)$. 
The same is of course  true, with $\ell=L$, if $u=\phi_\infty$. 
Conditions \eqref{eq:6} can be viewed as a kind of
 ``initial conditions'' at $r=\infty$. 
We show that equation \eqref{eq:w} with these conditions is well posed
and has analytic solutions.  Namely, we prove the following.
\begin{lemma}
  \label{le:shoot2}
  Given any  $\ell_0\in\{L\}\cup \{\ell(\al):\al\in \cS\setminus\{\kappa\}\}$
  and $r_0\in (0,\infty)$, there is $\theta>0$ and an analytic function
  $u:(r_0/2,\infty)\times (\ell_0-\theta,\ell_0+\theta)\to (0,\infty)$ with the
  following properties.
  \begin{itemize}
  \item[(i)] For any $\ell\in (\ell_0-\theta,\ell_0+\theta)$, the
  function $u(\cdot,\ell)$ is a positive solution of \eqref{eq:w} on
  $[r_0/2,\infty)$ satisfying \eqref{eq:6}, and it is the only
  solution (up to extensions and restrictions)   of
   \eqref{eq:w} satisfying \eqref{eq:6}. 
 \item[(ii)] The function  $u_\ell(\codt,\ell_0)$ can be extended to
   $(0,\infty)$, where it satisfies
   the linear equation
   \begin{equation}
  \label{eq:lina}
  z_{rr}+\left(\frac{N-1}r-\frac r2\right) z_r+\left(-\frac 1{p-1}+
    p\phi^{p-1}(r)\right)z=0,
\end{equation}
with $\phi=w(\codt,\al_0)$ if $\ell_0=\ell(\al_0)$ for some
$\al_0\in \cS\setminus\{\kappa\}$, and $\phi=\phi_\infty$ if
$\ell_0=L$.  Moreover, $r^{{2}/{(p-1)}}u_\ell(r,\ell_0)\to 1$ as
$r\to\infty$. 
  \end{itemize}
\end{lemma}
We prepare the proof of this lemma by transforming
the problem to one on a
bounded interval. First, setting
$v(r):=w(r)r^{2/(p-1)}$, we transform equation \eqref{eq:w} to
\eqref{eq:v}. Next, we set $y(\rho)=v(r)$, $\rho=1/r$. A simple
computation shows that $w$ is a solution of \eqref{eq:w} on
$(r_0,\infty)$ for some $r_0>0$
if and only if $y$ is a solution of the following
equation on $(0,1/r_0)$:
\begin{equation}
  \label{eq:y}
  y_{\rho\rho}+\frac{\ga}{\rho}y_\rho+\frac1{2\rho^3}y_\rho+
  \frac1{\rho^2}(y^p-{L}^{p-1}y)=0,
\end{equation}
with $\ga:=3-N+4/(p-1)$. 
Moreover, if conditions \eqref{eq:6} are satisfied by $u=w$, then, as
$\rho\downto 0$, one has  
$y(\rho)\to\ell$ and
\begin{equation*}
  y'(\rho)=-v'(r)r^2=-r^{\frac2{p-1}}
  w(r)r^2\left(\frac2{p-1}r^{-1}+\frac{w'(r)}{w(r)}\right)\to
  0\quad\text{as $r=\frac1\rho\to\infty$.} 
\end{equation*}
So $y$ extends to a $C^1$ function on $[0,1/r_0)$ with
\begin{equation}
  \label{eq:icy}
  y(0)=\ell, \quad y'(0)=0.
\end{equation}
Conversely, if $y$ is $C^1$ on $[0,1/r_0)$ and conditions
\eqref{eq:icy} hold, then $w$ is easily shown to satisfy \eqref{eq:6}.

To show that problem \eqref{eq:y}, \eqref{eq:icy} is well posed, we
write it in a an integral form. Define a function $H$ on $[0,\infty)$
by 
\begin{equation}
  \label{eq:8}
 H(0)=0,\quad  H(\rho)=\rho^\ga e^{-\rho^{-2}/4}\ \text{ if $ \rho>0$,}
\end{equation}
so that $H'(\rho)=\ga H(\rho)/\rho+  H(\rho)/(2\rho^3)$. Notice that
$H'(\rho)>0$ for all sufficiently small $\rho>0$. Equation
\eqref{eq:y} is equivalent to the following equation
\begin{equation}
  \label{eq:7}
 (H(\rho)y'(\rho))'+\rho^{-2}H(\rho)(y^p(\rho)-L^{p-1}y(\rho))=0.  
\end{equation}
Assuming $y$ satisfies \eqref{eq:icy}, we integrate \eqref{eq:7} to
obtain
\begin{align*}
  y'(\rho)&=-\frac1{H(\rho)}\int_0^\rho
              \eta^{-2}H(\eta)(y^p(\eta)-L^{p-1}y(\eta))\,d\eta\\
            &=-\frac2{H(\rho)}\int_0^\rho
              (\eta H(\eta))'-(\ga+1)H(\eta))(y^p(\eta)-L^{p-1}y(\eta))\,d\eta.
\end{align*}
After an integration by parts this becomes
\begin{equation}
  \label{eq:9}
  \begin{aligned}
    y'(\rho)&=-2\rho(y^p(\rho)-L^{p-1}y(\rho))\\
              &~\ \ +\frac{2(\ga+1)}{H(\rho)}\int_0^\rho
            H(\eta)(y^p(\eta)-L^{p-1}y(\eta))\,d\eta\\
            &~\ \  +\frac2{H(\rho)}\int_0^\rho
              (\eta H(\eta))(py^{p-1}(\eta)-L^{p-1})y'(\eta))\,d\eta.
\end{aligned}
\end{equation}
Conversely, noting that $H(\eta)/H(\rho)<1$ if $0<\eta<\rho<\de$ and
$\de>0$ is sufficiently small, one shows easily that if $y\in
C^1[0,\de]$, $y(0)=\ell$ and \eqref{eq:9} holds, then $y'(0)=0$ and
\eqref{eq:y} is satisfied.

We can now set up a suitable fixed point argument.  We work in the 
Banach space $X:=C([0,\de],\R^2)$ with a usual norm, say
$\|U\|=\|y\|_{L^\infty(0,\de)}+\|z\|_{L^\infty(0,\de)}$ for $U=(y,z)\in X$.
Let $U_0\in X$ stand for the constant function $(\ell_0,0)$. Fix any 
$\ep\in (0,\ell_0/2)$ and let $B$ stand for the open 
ball ($\bar B$ for 
the closed ball) in $X$ with center $U_0$ and radius $\ep$. 
Note that the choice of $\ep$ guarantees that
for any $(y,z)\in \bar B$ one has $y\ge \ell_0/2$.
For any $\ell$ sufficiently close to $\ell_0$, we consider
the map $\Psi^\ell:\bar B\to X$ defined by $\Psi^\ell(y,z)=(\tilde y,\tilde
z)$, where, for $\rho\in [0,\de]$, 
\begin{equation}
  \label{eq:10}
\begin{aligned}
  \tilde y(\rho)&=\ell+\int_0^\rho z(\eta)\,d\eta,\\
  \tilde z(\rho)&=-2\rho(y^p(\rho)-L^{p-1}y(\rho))
              +\frac{2(\ga+1)}{H(\rho)}\int_0^\rho
           H(\eta) (y^p(\eta)-L^{p-1}y(\eta))\,d\eta\\
           & ~\qquad \qquad~\qquad \qquad  +\frac2{H(\rho)}\int_0^\rho
              \eta H(\eta)(py^{p-1}(\eta)-L^{p-1})z(\eta))\,d\eta.
\end{aligned}
\end{equation}
Clearly, $y$ is a $C^1[0,\de]$-solution of \eqref{eq:y},
\eqref{eq:icy} if and only if $(y,y')$ is a fixed point of the map
$\Psi^\ell$.

\begin{lemma}
  \label{le:contr}
  If $\de$ and  $\theta$
  are sufficiently small positive numbers, then the map
$\Psi^\ell$ defined above is for each $\ell\in
(\ell_0-\theta,\ell_0+\theta)$ a $1/2$-contraction on $\bar B$.
Denoting its unique fixed point by $U^\ell$, the map $\ell\to U^\ell$
is an analytic $X$-valued map on $(\ell_0-\theta,\ell_0+\theta)$.
\end{lemma}
Before proving this lemma, we use it to complete the proof of Lemma
\ref{le:shoot2}. 

\begin{proof}[Proof of Lemma \ref{le:shoot2}]
  Lemma \ref{le:contr} and the notes preceding it 
  yield a positive 
  solution of \eqref{eq:v}, \eqref{eq:icy} on some interval
  $[r_1,\infty)$ and also imply the uniqueness of the solution and
  its analytic dependence on $\ell$. 
  Of course, as the equation has no singularity in $(0,\infty)$, we
  can combine these results with standard results from ordinary
  differential equations to prove the existence of an analytic
  function $u$ on  $(r_0/2,\infty)\times (\ell_0-\theta,\ell_0+\theta)$
 (with $\theta$ possibly smaller than in Lemma \ref{le:contr})
 such that statement (i) of Lemma \ref{le:shoot2} holds.

 Having proved that given any $r_0>0$ the function $u(r,\ell)$ is
 defined for $r\in [r_0/2,\infty)$ if $\ell$ is close enough to
 $\ell_0$, we see that $u_\ell(r,\ell_0)$ is defined for any $r\in (0,\infty)$.
 Differentiating the fixed point equation \eqref{eq:10} with respect
 to $\ell$ (using the smooth dependence of the fixed point on $\ell$)
 and reversing the transformations relating $y$ and $w$, we obtain
 that $r^{{2}/{(p-1)}}u_\ell(r,\ell_0)\to 1$ as $r\to\infty$.
 The regularity of the function $u$ allows us to differentiate 
 equation \eqref{eq:w}, with $w=u(\cdot,\ell)$, with respect
 $\ell$ to obtain the equation for $u_\ell(\cdot,\ell_0)$.
 This yields equation \eqref{eq:lina} with
 $\phi=u(\cdot,\ell_0)$. The uniqueness property of the solution
 $u$ implies that $u(\cdot,\ell_0)=w(\codt,\al_0)$
 if $\ell_0=\ell(\al_0)$ for some
$\al_0\in \cS\setminus\{\kappa\}$, and $u(\cdot,\ell_0)=\phi_\infty$ if
$\ell_0=L$. This completes the proof of Lemma \ref{le:shoot2}. 
\end{proof}

\begin{remark}
  \label{rm:higher} {\rm Clearly, we can differentiate equation
  \eqref{eq:w} with $w=u(\cdot,\ell)$ further to find equations for
  higher derivatives of $u(\cdot,\ell)$ with respect to $\ell$. For
  example, $u_{\ell\ell}(\cdot,L)$ is a solution of the following
  nonhomogeneous equation on $(0,\infty)$:
  \begin{equation} \label{eq:29aell}
    z_{rr}+\left(\frac{N-1}r-\frac r2\right) z_r+
    \left(-\frac 1{p-1}+
      p\phi_\infty^{p-1}(r)\right)z
    =-p(p-1)\phi_\infty^{p-2}(r)u_\ell^2(r,L).
  \end{equation} 
Note that the function
$r^{{2}/{(p-1)}}u_{\ell\ell}(r,L)=y_{\ell\ell}(1/r,L)$
stays bounded as $r\to\infty$. This observation will be useful in the
next section.  }
\end{remark}

\begin{proof}[Proof of Lemma \ref{le:contr}]
  As noted above,  $\ep<\ell_0/2$ guarantees that
  for any $(y,z)\in \bar B$ one has $y\ge \ell_0/2$.
  It follows that the maps
  \begin{equation}
    \label{eq:11}
    (y,z)\mapsto y^p, \quad (y,z)\mapsto y^{p-1}z
  \end{equation}
  are analytic $C[0,\de]$-valued maps on $B$.
  Note also that the map sending $u\in C[0,\de]$
  to the function  $\int_0^\rho H(\eta)/H(\rho)u(\eta)\,d\eta$ is a
  bounded linear operator on $C[0,\de]$. It follows that  the map
  $(\ell,U)\mapsto \Psi^\ell(U)$ is an analytic $X$-valued map on
  $(\ell_0-\theta,\ell_0+\theta)\times B$
  (the  smallness of $\theta$, $\de$ is not needed here).

  Choose $\de>0$ so small that $H'>0$ on $(0,\de)$. 
  Clearly, the maps \eqref{eq:11} are globally Lipschitz on $\bar
  B$. This and the relation $H(\eta)/H(\rho)<1$ for $0<\eta<\rho\le
  \de$ imply that, possibly after making $\de>0$ smaller,
  $\Psi^\ell:\bar B\to X$ is a $1/2$-contraction (for any $\ell$).

  We now show that if $\theta$ is sufficiently small and $\de$ is made
  yet smaller, if needed, then for each $\ell\in
  (\ell_0-\theta,\ell_0+\theta)$ one has $\Psi^\ell(\bar B)\subset
  \bar B$, that  is, $\Psi^\ell$ is a $1/2$-contraction on $\bar B$.

  To that aim, for any $U\in \bar B$ we estimate
  \begin{equation}
    \label{eq:12}
     \begin{aligned}
    \|\Psi^\ell(U)-U_0\|&=  \|\Psi^\ell(U_0)-U_0\|+
                          \|\Psi^\ell(U)-\Psi^\ell(U_0)\|\\
    &\le \|\Psi^\ell(U_0)-U_0\|+
      \frac1{2} \|U-U_0\|\\
     & \le \|\Psi^\ell(U_0)-U_0\|+
                         \frac\ep{2}.
  \end{aligned}
\end{equation}
Now
\begin{equation*}
  \Psi^\ell(U_0)(\rho)-U_0=\left(\ell-\ell_0,
   C_0 \big(-2\rho+2(\ga+1)\int_0^\rho\frac{
     H(\eta)}{H(\rho)}\,d\eta\big)\right),  
\end{equation*}
where $C_0=\ell_0^p-L^{p-1}\ell_0$. Clearly,
\begin{equation}
  \label{eq:13}
  \| \Psi^\ell(U_0)-U_0\|\le |\ell-\ell_0|+C\de\le \theta+C\de,
\end{equation}
where $C$ is determined by $C_0$ and $\ga$ (and is independent of
$\theta$ and $\delta$). Taking $0<\theta<\ep/4$ and making $\de>0$
smaller, if necessary, so that $C\de<\ep/4$, we obtain from
\eqref{eq:13}, \eqref{eq:12} that $\| \Psi^\ell(U)-U_0\|<\ep$---that
is, $\Psi^\ell(U)\in \bar B$---for any $U\in \bar B$.

The uniform contraction theorem implies the existence of
a unique fixed point $U^\ell$ of $\Psi^\ell$, and
it also gives the analyticity of the map  $\ell\to
U^\ell:(\ell_0-\theta,\ell_0+\theta)\to X$. 
\end{proof}

Although not needed below, we add a remark on the dependence of the solutions on $p$. Clearly,
when dealing with solutions bounded below by a positive constant, one can view $p$ as a parameter,
with the nonlinearity $w^p$ depending analytically on $p$. Therefore the uniform contraction arguments
employed in the shooting from $0$ and $\infty$ imply that the solutions $w(\cdot,\al)$, $u(\cdot,\ell)$ given by
Lemmas \ref{le:shoot1}, \ref{le:shoot2} depend analytically on $p$, too.

\subsection{The discreteness of the set  $\cS$}
\label{discreteness}
We now show that the set $\cS$ is discrete, hence at most
countable. This will prove  statement (i) of Theorem \ref{thm:main}.

We go by contradiction. Suppose that $\cS$ contains an element
which is not
isolated in $\cS$. Set
\begin{equation}
  \label{eq:14}
  \al_0:=\inf\{\al\in \cS: \al \text{ is an accumulation point of }\cS\}.
\end{equation}
Clearly, $\al_0$ itself is an accumulation point of $\cS$.
By the continuity of the solutions $w(\codt,\al)$
with respect to $\al$, one has $\al_0\in \cS$.  By Lemma
\ref{le:basic}(i),  $\al_0>\kappa$. Set $\ell_0:=\ell(\al_0)$ (cp.
Lemma \ref{le:basic}(ii)).

Choose $\ep>0$ and $\theta>0$ such that the function $w(\cdot,\al)$ is
positive  on $(0,2)$ for all $\al\in (\al_0-\ep,\al_0+\ep)$ and the
function $u(\cdot,\ell)$ is (defined and is) positive on $(1,\infty)$
for all $\ell\in (\ell_0-\theta,\ell_0+\theta)$ 
(see Lemmas  \ref{le:shoot1}, \ref{le:shoot2}).
Recalling from  Section \ref{shoot1} and Lemma \ref{le:shoot2}
that the functions $w_\al(\cdot,\al_0)$, $u_\ell(\cdot,\ell_0)$ are
nontrivial solutions of the linear equation \eqref{eq:lin},
we pick $r_0\in (1,2)$ such that neither of these functions
vanishes at $r_0$. Then, making $\ep>0$ and $\theta>0$ smaller if necessary,
we may assume that
\begin{equation}
  \label{eq:16}
  \begin{aligned}
    w_\al(r_0,\al)&\ne 0\quad (\al\in (\al_0-\ep,\al_0+\ep)),\\
    u_\ell(r_0,\ell)&\ne 0\quad (\ell\in (\ell_0-\theta,\ell_0+\theta)).
  \end{aligned}
\end{equation}

Now, Lemma \ref{le:cont} guarantees that, possibly after 
making $\ep>0$ yet smaller, one has $\ell(\al)\in
(\ell_0-\theta,\ell_0+\theta)$ for any $\al\in
(\al_0-\ep,\al_0+\ep)\cap\cS$. For any such $\al$, Lemmas
\ref{le:basic}(ii) and \ref{shoot2} imply that
$w(\cdot,\al)\equiv u(\codt,\ell(a))$; in particular,
\begin{equation}
  \label{eq:15}
  (w(r_0,\al),w_r(r_0,\al))=(u(r_0,\ell(\al),u_r(r_0,\ell(\al))\quad (\al\in
(\al_0-\ep,\al_0+\ep)\cap\cS).
\end{equation}

Consider the following two
analytic curves
\begin{align*}
  J_1&:=\{(w(r_0,\al),w_r(r_0,\al)):\al\in (\al_0-\ep,\al_0+\ep)\},\\
  J_2&:=\{(u(r_0,\ell),u_r(r_0,\ell)):\ell\in (\ell_0-\theta,\ell_0+\theta)\}.
\end{align*}
In view of \eqref{eq:16}, they
can be reparameterized by the first component, namely,
\begin{align*}
  J_1&:=\{(\zeta,F(\zeta)):\zeta\in I_1\},\\
  J_2&:=\{(\zeta, G(\zeta)):\zeta \in I_2\},
\end{align*}
where $I_1$ is the open interval with the end points
$w(r_0,\al_0\pm\ep)$, $I_2$ is the open interval with the end points
$u(r_0,\ell_0\pm\theta)$, and $F$ and $G$ are  analytic
functions: $F(\zeta)= w_r(r_0,\hat\al(\zeta))$, where $\hat\al$ is the
inverse to $\al\mapsto w(r_0,\al)$; and similarly for $G$.
Since $\ell_0=\ell(\al_0)$, relation \eqref{eq:15} implies that
$w(r_0,\al_0)=u(r_0,\ell_0)=:\zeta_0\in I_1\cap I_2$ and
$F(\zeta_0)-G(\zeta_0)=0$. Further, using  \eqref{eq:15} in
conjunction with the fact $\al_0$ is an accumulation point of $\cS$,
we obtain that $\zeta_0$ is an accumulation point of the set of zeros
of the function $F-G$. By the analyticity, $F-G$ vanishes identically
on a neighborhood of $\zeta_0$. From this and the relation
$w(\cdot,\al) \equiv u(\cdot,\ell(\al))$, we conclude that
for $\al$ in a neighborhood of $\al_0$ the solution $w(\cdot,\al)$ is
positive on $(0,\infty)$, that is, $\al\in \cS$. This is a
contradiction to the definition of $\al_0$ (cp. \eqref{eq:14}).
With this contradiction, the discreteness of $\cS$ and statement
(i) of Theorem \ref{thm:main} are proved.

\section{Solutions near $\phi_\infty$ and
  the proof of Theorem \ref{thm:main}(ii)}
\label{proof2}
In this section we assume  that $p_{JL}<p\le p_{L}$.

Our goal is to show that there
is a constant $\al^*>0$ such that
\begin{equation}
  \label{eq:20}
  \cS\cap (\al^*,\infty)=\emptyset.
\end{equation}
In conjunction with statement (i) of Lemma \ref{le:basic} and
the discreteness of the set $\cS$ proved in the
previous section, \eqref{eq:20} implies that the set $\cS$ is finite.
Thus, once we prove \eqref{eq:20}, the proof of
statement (ii) of Theorem \ref{thm:main}
will be complete.

To prove \eqref{eq:20}, we initially employ the functions $w(r,\al)$,
$u(r,\ell)$ in a very similar manner as in Section \ref{discreteness},
taking $\ell$  close to the constant $L$ from the singular solution
(cp. \eqref{eq:3}). 

First we choose 
$\al^*>0$ and $R>0$ such that
\begin{equation*}
 w_\al(r,\al)>0, \ w(r,\al)>0  \quad(r\in [0,R]),
\end{equation*}
and \eqref{eq:242} holds for some constant $C>0$
(cp. Lemma \ref{le:pos}, \ref{le:large}).
  Next we choose $\theta>0$ such that the 
function $u(\cdot,\ell)$ is (defined and) positive on $(R/2,\infty)$
for all $\ell\in (L-\theta,L+\theta)$.  Pick $r_0\in
(R/2,R)$ such that  $u_\ell(r_0,L)\ne 0$.  Making $\theta>0$ smaller
if necessary, we have
\begin{equation} \label{eq:21}
    u_\ell(r_0,\ell)\ne 0\quad (\ell\in (L-\theta,L+\theta)).
\end{equation}
Further, by Lemma \ref{le:cont}(ii), we have, possibly after 
making $\al^*$ larger, that $\ell(\al)\in
(L-\theta,L+\theta)$ for any $\al\in
(\al^*,\infty)\cap\cS$. For any such $\al$, Lemmas
\ref{le:basic}(ii) and \ref{shoot2} imply that
$w(\cdot,\al)\equiv u(\codt,\ell(a))$; in particular,
\begin{equation}
  \label{eq:22}
  (w(r_0,\al),w_r(r_0,\al))=(u(r_0,\ell(\al),u_r(r_0,\ell(\al))\quad (\al\in
(\al^*,\infty)\cap\cS).
\end{equation}
Consider the following two
analytic curves
\begin{align*}
  J_1&:=\{(w(r_0,\al),w_r(r_0,\al)):\al\in (\al^*,\infty)\},\\
  J_2&:=\{(u(r_0,\ell),u_r(r_0,\ell)):\ell\in (L-\theta,L+\theta)\}.
\end{align*}
In view of the relations $w_\al(r_0,\al)>0$ and \eqref{eq:21}, using
also the fact that $w(r_0,\al)\to\phi_\infty(r_0)=:\zeta_0$ as
$\al\upto\infty$ (cp. Lemma \ref{le:pos}),
 we reparameterize the curves $J_1$, $J_2$ as follows: 
\begin{align}
  J_1&:=\{(\zeta,F(\zeta)):\zeta\in (w(r_0,\al^*), \zeta_0)\},\\
  J_2&:=\{(\zeta, G(\zeta)):\zeta \in I\}.
\end{align}
Here $I$ is the open interval with the end points
$u(r_0,L\pm\theta)$, and $F$ and $G$ are analytic
functions: $F(\zeta)= w_r(r_0,\hat\al(\zeta))$, where $\hat\al$ is the
inverse to $\al\mapsto w(r_0,\al)$; and, similarly,
$G(\ze)= u_r(r_0,\hat\ell(\zeta))$, where $\hat\ell$ is the
inverse to $\ell\mapsto u(r_0,\ell)$.

Since $u(\cdot,L)=\phi_\infty$, we have 
$\zeta_0\in I$ and $G(\zeta_0)=\phi_\infty'(r_0)$. Also,
from the fact that $w_r(r_0,\al)\to \phi_\infty'(r_0)$ as
$\al\to\infty$ (cp. \eqref{eq:5}), we infer that $\lim_{\zeta\upto
  \zeta_0}F(\zeta)= \phi_\infty'(r_0)$. Thus, we may define
$F(\zeta_0):=\phi_\infty'(r_0)$ and $F$ becomes a continuous function
on $(w(r_0,\al^*), \zeta_0]$.

If $F$ were analytic on the interval $(w(r_0,\al^*), \zeta_0]$,
we could use simple analyticity arguments, similar to those in Section
\ref{discreteness}, to conclude the proof of \eqref{eq:20}. However,
it turns out that in some cases $F$
is not even of class $C^2$ at $\zeta_0$, and we thus need a different
reasoning.

We will prove the following statements.
\begin{proposition}
  \label{prop:derivatives} Let $F$ and $G$ be as above.
  Then the function $F$  is of class $C^1$ on
  $(w(r_0,\al^*), \zeta_0]$
  and the following statements hold:
  \begin{itemize}
  \item[(i)] $\la=0$ is 
    an eigenvalue of problem \eqref{eq:17} if and only if
    $F'(\zeta_0)= G'(\zeta_0)$.
  \item[(ii)] If $\la=0$ is
    an eigenvalue of problem \eqref{eq:17}, then 
    $\lim_{\zeta\to\zeta_0}F''(\zeta)$ exists and is distinct from
    $G''(\zeta_0)$. More specifically, the following
    statements are valid (with $\be$ as in \eqref{eq:19}):
    \begin{itemize}
    \item[(a)] If
      \begin{equation}
  \label{eq:43nn}
  N-1+3\be-\frac{2(p-2)}{(p-1)}\le-1,
     \end{equation}
then 
    \begin{equation}
      \label{eq:23nn}
      F''(\zeta)\to-\infty\quad\text{as $\zeta\upto \zeta_0$}.
    \end{equation}
 \item[(b)] If \eqref{eq:43nn} is not true, then  $F''(\zeta)$ has a
   finite limit as $\zeta\to\zeta_0$ and
   \begin{equation}
     \label{eq:49}
     \lim_{\zeta\to\zeta_0}F''(\zeta)\ne G''(\zeta_0).
   \end{equation}    
  \end{itemize}
    \end{itemize}
\end{proposition}
\begin{remark}
  \label{rm:pj}{\rm
    \begin{itemize}
    \item[(i)] 
      It may be instructive---and will be useful below---to
      list the exponents $p>p_{JL}$  for which $\la=0$ is
    an eigenvalue of problem \eqref{eq:17}. These can be computed from
    \eqref{eq:18}, \eqref{eq:19}: assuming $N>10$,
    for $j\ge 2$ we have $\la_j=0$ if and only if
    $p=p_j$, where
    \begin{equation}
      \label{eq:46}
      p_j:=1+\frac{4j-2}{N(j-1)-2j^2-2j+2}.
    \end{equation}
   As already mentioned in Lemma \ref{le:HV},
   $\la_2=0$ for $p=p_L$ (in other words, $p_2=p_L$),
   and $\la_2<0$ for any $p_{JL}<p<p_L$. 
   To have $p_j> p_{JL}$ for some $j\ge 3$, $N$ has to be sufficiently
   large. Specifically,  $p_j> p_{JL}$ if and only if $N>
   (2j-1)^2+1$. Thus, for  example, 
   if $N\le 26$, then $\la=0$ is not
   an eigenvalue of problem \eqref{eq:17}
   for any $p\in (p_{JL},p_L)$;
   if $26<N\le 50$, it is an eigenvalue for exactly one
   $p\in (p_{JL},p_L)$, namely $p=p_3$; and so on.
 \item[(ii)] Assume $p\in (p_{JL},p_L)$. 
   As noted in the previous remark,
   the assumption
   of statement (ii) of Proposition \ref{prop:derivatives} ($\la=0$ being an
   eigenvalue of \eqref{eq:17}) 
    is void if $N\le 26$. 
    Also, if  $26<N\le 50$ and the assumption is satisfied,
   then necessarily  $p= p_3$ (and $j=3$).  In this case, 
   condition \eqref{eq:43nn} is automatically satisfied.
   This follows from the
   relations  \eqref{eq:46} and $\be=-{2}/{(p-1)}-6$
   (cp. \eqref{eq:18}).
   However, for larger dimensions, \eqref{eq:43nn} is not always
   satisfied. For example, in the case of $p=p_3$ (when $\la_3=0$), 
   \eqref{eq:43nn} is not satisfied if 
   $N> 56$. 
    \end{itemize}
    }
\end{remark}

Before proving Proposition \ref{prop:derivatives} , we show how it implies \eqref{eq:20}.
\begin{proof}[Proof of \eqref{eq:20}]
  Recall that the function $G$ is analytic in a neighborhood of
  $\zeta_0$. 
  Proposition \ref{prop:derivatives} implies that
  either $F'(\zeta_0)\ne G'(\ze_0)$ or there exists $\zeta_1<\zeta_0$ such that
  $F''(\zeta)\ne G''(\zeta)$ for all $\zeta\in
  (\zeta_1,\zeta_0)$. In either case, $\zeta_0$ is clearly not an
  accumulation point of the set of zeros of the function
   $F-G$. This is equivalent to \eqref{eq:20}.
\end{proof}
\begin{remark}
  \label{rm:ketc}{\rm
There is a strong indication (see Remark \ref{rm:last} below)
   that whenever $\la=0$ is 
    an eigenvalue of problem \eqref{eq:17}, then there is an integer
    $k\ge 2$ such that
    \begin{equation}
      \label{eq:23}
      |F^{(k)}(\zeta)|\to\infty\quad\text{as $\zeta\upto \zeta_0$}.
    \end{equation}
    If confirmed, this could be used---instead of statement
    (ii)(b) of Proposition \ref{prop:derivatives}---as an alternative 
    proof of \eqref{eq:20} (the arguments would be
    similar as with $k=2$ in the case (ii)(a)). 
    }
\end{remark}

The rest of the section devoted to the proof of Proposition
\ref{prop:derivatives}. 
We carry out the proof in several steps. In some cases, we do the
computations in greater generality than needed for the proof, as these
may be of some interest and do not require much extra work.  

\medskip
\noindent
\emph{STEP 1: {Computation of the derivatives $F^{(k)}(\zeta)$,
    $G^{(k)}(\zeta)$.}}
\smallskip

Recall that, assuming $\al$ is sufficiently large, we have
$F(\zeta)= w_r(r_0,\hat\al(\zeta))$, where $\hat\al$ is the
inverse to $\al\mapsto w(r_0,\al)$. Therefore, 
we have
\begin{equation}
  \label{eq:25}
  F'(\zeta)=\frac{w_{r\al}(r_0,\al)}{w_{\al}(r_0,\al)},\quad\text{with }
  \al=\hat\al(\zeta).\
\end{equation}
Similarly, for any integer $k>1$,
if $F^{(k)}(\zeta)=:g(\al)$ for $\al=\hat\al(\zeta)$, then
\begin{equation*}
  F^{(k+1)}(\zeta)=\frac1{w_{\al}(r_0,\al)}\partial_\al g(\al).
\end{equation*}
Hence, by induction, for $k=1,2,\dots$ we have
\begin{equation}
  \label{eq:261}
  F^{(k)}(\zeta)=(\hat \partial_\al)^kw_r(r_0,\al),\quad
  \text{with } \al=\hat\al(\zeta),
\end{equation}
where $\hat \partial_\al$ is a differential operator given by
\begin{equation*}
 \hat \partial_\al:= \frac1{{w_{\al}(r_0,\al)}}{\partial_\al}.
\end{equation*}
Set
\begin{equation}
  \label{eq:26}
  \psi(r,\al):=\hat\partial_\al w(r,\al)=\frac{w_{\al}(r,\al)}{w_{\al}(r_0,\al)}.
\end{equation}
Note that $\psi(\codt,\al)$ is a
solution of
the following problem with a homogeneous differential equation:
 \begin{align}  \label{eq:28a}
     z_{rr}+\left(\frac{N-1}r-\frac r2\right) z_r+\left(-\frac 1{p-1}+
       p(w(r,\al))^{p-1}\right)z&=0,\quad r\in (0,R],\\
     z(r_0)=1, \quad \text{$z(r)$ is bounded as $r\downto 0$.}  \label{eq:28b}
 \end{align}
Similarly, for $k= 2,3, \dots$, $\hat\partial_\al^kw(r,\al)$
is a solution of the following problem 
with a nonhomogeneous differential equation:
\begin{align} \label{eq:29a}
     z_{rr}+\left(\frac{N-1}r-\frac r2\right) z_r+&\left(-\frac 1{p-1}+
       p(w(r,\al))^{p-1}\right)z=f_k(r,\al),\ \, r\in (0,R],\\
     z(r_0)&=0, \quad \text{$z(r)$ is bounded as }r\downto 0, \label{eq:29b}
  \end{align}
where
\begin{equation}
  \label{eq:30}
  f_k(r,\al)=p(w(r,\al))^{p-1}\hat\partial_\al^kw(r,\al)-\hat\partial_\al^kw^p(r,\al).
\end{equation}   
We remark that \eqref{eq:30} is just a compact way of writing the
right-hand side. The $k$-derivative actually cancels out in
\eqref{eq:30} so $f_k$ depends on lower derivatives only. Obviously,
for each fixed $\al$, the function $f_k(r,\al)$ is bounded as
$r\to0$.

For the function $G$ and $k=1,2\dots$, we have similarly as
for $F$ in \eqref{eq:25},\eqref{eq:261},
\begin{equation}
  \label{eq:53}
   G^{(k)}(\zeta)=(\tilde\partial_\ell)^ku_{r}(r_0,\ell),\quad\text{with
   }\ell=\hat\ell(\zeta),   
\end{equation}
where
\begin{equation}
  \label{eq:54}
  \tilde\partial_\ell:=\frac1{u_\ell(r_0,\ell)}\partial_\ell.
\end{equation}

\medskip
\noindent
\emph{STEP 2: Relation of $F'(\zeta_0)$ to $G'(\zeta_0)$ and the proof of statement
  (i) of Proposition \ref{prop:derivatives}.   }
\smallskip

We find the (left) derivative $F'(\zeta_0)$
using the definition of $F$ and the L'Hospital rule:
  \begin{equation}
    \label{eq:34}
    \begin{aligned}
  \lim_{\zeta\upto\zeta_0}\frac{F(\zeta)-F(\zeta_0)}{\zeta-\ze_0}&=
  \lim_{\al\to\infty}
  \frac{w_r(r_0,\al)-\phi_\infty'(r_0)}{w(r_0,\al)-\phi_\infty(r_0)}\\
  &=      
  \lim_{\al\to\infty}\frac{w_{\al r}(r_0,\al)}{w_\al(r_0,\al)}=
  \lim_{\al\to\infty}\psi_r(r_0,\al),
\end{aligned}
\end{equation}
where $\psi$ is as in \eqref{eq:26}. Using Lemma \ref{le:large},
the uniform bound \eqref{eq:242}, and  regularity properties of solutions
of linear differential equations,
one shows easily
that, as $\al\to\infty$,
\begin{equation}\label{eq:39}
\begin{aligned}
  \psi(\cdot,\al)&\to \psi_\infty,\\
  \psi_r(\cdot,\al)&\to \psi_\infty', 
\end{aligned}
\end{equation}
with the convergence in $L_{loc}^\infty(0,\infty)$, where
$\psi_\infty$ is a positive
solution of equation \eqref{eq:170} satisfying the following relations for
some positive constant $C_1$:
\begin{equation}
  \label{eq:33}
\psi_\infty(r_0)=1,\qquad\psi_\infty(r) \le C_1r^{\be}\quad (r\in
      [0,r_0]) 
    \end{equation}
    ($\be$ is as in \eqref{eq:19}).
The solution $\psi_\infty$ is uniquely determined. In fact, it follows
follows from \eqref{eq:33} that
\begin{equation}
  \label{eq:38}
  \psi_\infty\equiv c_0\psi_1\ \text{ with $c_0:=\frac1{\psi_1(r_0)}$,}
\end{equation}
where $\psi_1$ is as in Lemma  \ref{le:anal}.

Thus, the limit in
\eqref{eq:34} exists and we have   
$F'(\zeta_0)=\psi_\infty'(r_0)$.      
Now,  by \eqref{eq:261}, \eqref{eq:26}, we also have 
\begin{equation}
  \label{eq:35}
  \lim_{\zeta\upto\zeta_0}F'(\zeta)=
  \lim_{\al\to\infty}\psi_r(r_0,\al)=\psi_\infty'(r_0),
\end{equation}
showing that $F$ is of class $C^1$ on $(w(r_0,\al^*), \zeta_0]$.

The derivate $G'(\zeta_0)$ is obtained directly from \eqref{eq:53},
\eqref{eq:54} using  the relations $u(r_0,L)=\phi_\infty(r_0)=\zeta_0$:
\begin{equation*}
  G'(\zeta_0)=\frac{u_{r\ell}(r_0,L)}{u_{\ell}(r_0,L)}.
\end{equation*}
By Lemma \ref{le:shoot2},
$\hat\psi(r):={u_{\ell}(r,L)}/{u_{\ell}(r_0,L)}$ is a solution of
\eqref{eq:170} satisfying
$\hat\psi(r_0)=1$ and
\begin{equation}
  \label{eq:24}
  \text{$r^{{2}/{(p-1)}}\hat\psi(r)\to \frac1{u_{\ell}(r_0,L)}$ \ as
$r\to\infty$.}
\end{equation}

We now complete the proof of statement (i) of Theorem \ref{thm:main}.
The relation $F'(\ze_0)=G'(\zeta_0)$ is equivalent to
$\psi_\infty'(r_0)=\hat \psi'(r_0)$. Since also 
$\psi_\infty(r_0)=1=\hat \psi(r_0)$ and $\psi_\infty$, $\hat\psi$ are
solutions of \eqref{eq:170}, the relation $F'(\ze_0)=G'(\zeta_0)$
is actually equivalent to the identity
$\hat \psi \equiv \psi_\infty $.
In view of 
\eqref{eq:38} and \eqref{eq:24},
the identity means that 
the solution
$\psi_1$ in Lemma \ref{le:anal} satisfies \eqref{eq:31}.
By Lemma \ref{le:anal}, this is equivalent to 
$\la=0$ being an eigenvalue of \eqref{eq:17}.
Statement (i) is proved.

\medskip
\noindent
\emph{STEP 3: Variation of constants and an integral formula for
$F^{(k)}(\zeta)$.}
\smallskip

We find a tangible formula
for the functions $F^{(k)}(\zeta)$, $k=2,3,\dots$. 
For a while, we will consider $\al>\al^*$ fixed and  write $\psi$ for
$\psi(\codt,\al)$, $f_k$ for $f_k(\cdot,\al)$. Remember that
$\psi$ is a solution of
\eqref{eq:28a}, \eqref{eq:28b}.
Let $\varphi$ be the solution of 
\eqref{eq:28a} with
\begin{equation}
  \label{eq:32}
  \varphi(r_0)=0,\quad \varphi'(r_0)=-\frac{1}{\omega(r_0)},
\end{equation}
where $\omega$ is defined in \eqref{eq:anew}.

Obviously, $\psi$, $\varphi$ are linearly independent. We claim that
for some constant $c\ne 0$ one has
\begin{equation}
  \label{eq:36}
  \varphi(r)=cr^{-(N-2)}+o(r^{-(N-2)})\ \text{ as $r\to0$.}
\end{equation}
In fact, any solution $\varphi$ 
linearly independent from $\psi$ has this property.
One way to see this is by using the Frobenius method.
Observe that multiplying equation \eqref{eq:28a} by $r^2$,
we obtain an equation with analytic coefficients (near $r=0$) and a
regular singular point at $r=0$. We look for solutions in the form of
a convergent Frobenius series
\begin{equation}
  \label{eq:37}
  z(r)=r^\vartheta\sum_{j=0}^\infty c_jr^j,
\end{equation}
where $c_j$ are real coefficients, $c_0\ne0$,
and $\vartheta$ is a root of the 
indicial equation
\begin{equation*}
  \vartheta(\vartheta-1)+(N-1)\vartheta=0.
\end{equation*}
The larger root $\vartheta=0$ always yields solutions of the form
\eqref{eq:37}; such solutions are bounded near $0$ and
they are all scalar multiples of $\psi$. Now, since the smaller root,
$\vartheta :=-(N-2)$, is also an integer, 
the linearly independent solution $\varphi$ is
either given by \eqref{eq:37} (with $\vartheta =-(N-2)$), or by the
formula
\begin{equation}
  \label{eq:37nn}
  \varphi(r)= C\psi(r)\log r+z(r)
\end{equation}
where $z$ is as in \eqref{eq:37} with $c_0\ne0$ and $C\in \R$
(possibly $C=0$), see \cite[Theorem~4.5]{Teschl}, for example.  
In either case, \eqref{eq:36} holds.

We use the linearly independent solutions $\psi$, $\varphi$ in 
the variation of constants formula. 
The homogeneous equation \eqref{eq:28a} can be written as
\begin{equation*}
  (\omega(r)z_r)_r+ \omega(r)\left(-\frac 1{p-1}+
       pw^{p-1}(r,\al)\right)z=0. 
\end{equation*}
The Wronskian of the solutions $\psi$, $\varphi$, that is, the function
\begin{equation*}
  W(r):=\psi'(r)\varphi(r)-\psi(r)\varphi'(r),
\end{equation*}
satisfies
$ (\omega(r)W(r))'=0$ for all $r>0$ (as long as $w(\cdot,\al)$ stays positive)
and $\omega(r_0)W(r_0)=1$ (cp. \eqref{eq:28b},
\eqref{eq:32}). So $W(r)=1/\omega(r)$ for all $r>0$. A
standard variation of constants formula (easily verified by 
direct differentiation) yields the general solution
of \eqref{eq:29a}:
    \begin{multline}
      \label{eq:42}
        z(r)=\left(c_1-\int_r^{r_0}\omega(s)f_k(s)\varphi(s)\,ds\right)\psi(r)
\\        +
  \left(c_2+\int_r^{r_0}\omega(s)f_k(s)\psi(s)\,ds\right)\varphi(r).
\end{multline}
Here $c_1, c_2\in\R$ are arbitrary parameters. 
For \eqref{eq:42} to give a solution with $z(r_0)=0$,  it is necessary
and sufficient that $c_1=0$. For this solution to also be bounded
as $r\to0$, it is necessary that
\begin{equation*}
  c_2+\int_0^{r_0}\omega(s)f_k(s)\psi(s)\,ds=0.
\end{equation*}
This follows from the boundedness of $\psi$, $f_k$, and formulas
\eqref{eq:36}, \eqref{eq:anew}. Thus, we get 
\begin{equation}
      \label{eq:42n}
        z(r)=-\psi(r) \int_r^{r_0}\omega(s)f_k(s)\varphi(s)\,ds  -
  \varphi(r)\int_0^{r}\omega(s)f_k(s)\psi(s)\,ds,
\end{equation}
showing in particular that the solution of \eqref{eq:29a}, \eqref{eq:29b}
is unique.  Using 
\eqref{eq:36}, \eqref{eq:anew},
one verifies easily that the function $z$ given
by \eqref{eq:42n} is
bounded near $r=0$, so it is the unique
solution of \eqref{eq:29a}, \eqref{eq:29b}.   

Differentiating  \eqref{eq:42n} and using \eqref{eq:32}, we obtain
\begin{equation}
  \label{eq:44}
  z'(r_0)=\frac1{\omega(r_0)}\int_0^{r_0}\omega(s)f_k(s)\psi(s)\,ds.
\end{equation}

We now summarize the above computations, bringing back the $\al$-variable.
Using \eqref{eq:261}, \eqref{eq:29a}, \eqref{eq:29b}, and substituting from
\eqref{eq:anew}, we obtain
that for $k=2,3,\dots$ 
\begin{equation}
  \label{eq:28}
  F^{(k)}(\zeta)=\frac1{\omega(r_0)}\int_0^{r_0}s^{N-1}e^{-{s^2}/{4}}f_k(s,\al)\psi(s,\al)\,ds,
\quad \text{with } \al=\hat\al(\zeta),
\end{equation}
where
\begin{equation}
  \label{eq:30n}
  f_k(r,\al)=pw^{p-1}(r,\al)\hat\partial_\al^kw(r,\al)-\hat\partial_\al^kw^p(r,\al).
\end{equation} 

\medskip
\noindent
\emph{STEP 4: Estimates of  $F''(\zeta)$ 
  as $\zeta\upto \zeta_0$ and the proof of statement (ii )(a)
  of Proposition \ref{prop:derivatives}. }  
\medskip

 For $k=2$, formulas \eqref{eq:28}, \eqref{eq:30n}, \eqref{eq:26} give
 \begin{equation}
  \label{eq:282}
  F^{''}(\zeta)=-\frac{p(p-1)}{\omega(r_0)}\int_0^{r_0}s^{N-1}e^{-{s^2}/{4}}
  (w(s,\al))^{p-2}\psi^3(s,\al)\,ds\quad(\al=\hat\al(\zeta)). 
\end{equation}
Recall that $\psi^3(\cdot,\al)>0$ on $[0,R)\superset [0,r_0]$ for all
$\al>\al^*$.

We have $\al\to\infty$ as $\zeta\to\zeta_0$.
Also, for any $s\in (0,r_0]$, 
\begin{equation}
\label{eq:29}
\begin{aligned}  
 \lim_{\al\to\infty}
  (w(s,\al))^{p-2}\psi^3(s,\al)&=(\phi_\infty(s))^{p-2}(\psi_\infty(s))^3\\
  &=
 L^{p-2}s^{-2(p-2)/(p-1)}  (\psi_\infty(s))^3
\end{aligned}
\end{equation}
(see Step 1). Now, by \eqref{eq:38}, $\psi_\infty$ is a nonzero scalar
multiple of the
solution $\psi_1$ in Lemma \ref{le:anal}. Moreover, being the limit of 
$\psi(\cdot,\al)$, $\psi_\infty$ is nonnegative (hence positive) in
$(0,r_0]$. Therefore, by \eqref{eq:30a}, there is a
positive constant $c_1$ such that
\begin{equation}
  \label{eq:40}
 c_1s^{3\beta}\le (\psi_\infty(s))^3\le c_1^{-1}s^{3\beta} \quad (s\in (0,r_0)).
\end{equation}

Assume now that \eqref{eq:43nn} is true. 
It follows from \eqref{eq:282}--\eqref{eq:40} and  Fatou's lemma that
for some positive constant $c_2$ one has
\begin{equation}
  \label{eq:41}
  \limsup_{\zeta\upto\zeta_0} F''(\zeta)\le -c_2\int_0^{r_0}
  s^{N-1+3\be-2(p-2)/(p-1)}\,ds. 
\end{equation}
By \eqref{eq:43nn}, the integral is infinite, hence \eqref{eq:49} holds.
With this we have completed the proof statement (ii)(a) of Proposition
\ref{prop:derivatives}.

Returning to \eqref{eq:282}, we now find the limit of $F''(\zeta)$
assuming \eqref{eq:43nn} is not true, that is,  
 \begin{equation}
  \label{eq:43nnot}
 \ga:= N-1+3\be-\frac{2(p-2)}{(p-1)}>-1.
\end{equation}
Using the relations \eqref{eq:242} and $w(r,\al)<\phi_\infty(r)$, we
find an upper bound on the  integrand in  \eqref{eq:282} in the form
$c s^\ga$, where $c>0$ is a constant. By 
\eqref{eq:43nnot}, this is an integrable function, hence 
\eqref{eq:29} and the Lebesgue dominated convergence theorem 
yield the finite limit
\begin{equation}
  \label{eq:43}
  \lim_{\zeta\upto\zeta_0} F''(\zeta)=-\frac{p(p-1) L^{p-2}}{\omega(r_0)}
  \int_0^{r_0}s^{N-1-2(p-2)/(p-1)}e^{-{s^2}/{4}}(\psi_\infty(s))^3
  \,ds.
\end{equation}

\medskip
\noindent
\emph{STEP 5: Fredholm alternative and  the proof of statement (ii)(b)
  of Proposition \ref{prop:derivatives}. }  
\medskip

We assume that $\la=0$ is an eigenvalue of \eqref{eq:17} (that is,
$F'(\zeta_0)=G'(\zeta_0)$) and \eqref{eq:43nnot} holds.
Thus the limit $\lim_{\zeta\upto\zeta_0} F''(\zeta)$ is given
by \eqref{eq:43} and it is finite. 
We  claim
 that for the relation 
\begin{equation}
  \label{eq:50}
  \lim_{\zeta\upto\zeta_0} F''(\zeta)=G''(\zeta_0)
\end{equation}
to hold it is necessary that
\begin{equation}
  \label{eq:45}
   \int_0^{\infty}s^{N-1-2(p-2)/(p-1)}e^{-{s^2}/{4}}\psi^3_\infty(s)
  \,ds=0.
\end{equation}

We first prove this claim and then verify that \eqref{eq:45} does not
hold. This will prove statement (ii)(b) and complete the proof 
  of Proposition \ref{prop:derivatives}.

To prove the claim, assume \eqref{eq:50}. According to
\eqref{eq:261},  the finite limit in \eqref{eq:50} is also the limit
of $(\hat\partial_\al)^2w_r(r_0,\al)$ as $\al\to\infty$. Therefore,
using the condition $(\hat\partial_\al)^2w(r_0,\al)=0$ (cp.
\eqref{eq:29b}) and taking the limit in equation \eqref{eq:29a} with
$k=2$, we infer that, as $\al\to\infty$,
\begin{equation*}
  z_2(\cdot,\al):=(\hat\partial_\al)^2w(\cdot,\al)\to
  z_2^\infty\quad\text{in }C^1_{loc}(0,\infty)
\end{equation*}
where $z_2^\infty$ is the solution of the initial value
problem
\begin{align} \label{eq:29ann}
     z_{rr}+\left(\frac{N-1}r-\frac r2\right) z_r+&\left(-\frac 1{p-1}+
    p\phi_\infty^{p-1}(r)\right)z= f_2^\infty(r) ,\\
     z(r_0)&=0, \quad z'(r_0)=G''(\zeta_0), \label{eq:29bnn}
\end{align}
with
\begin{equation}
  \label{eq:52}
  f_2^\infty(r):= -p(p-1)\phi_\infty^{p-2}(r) \psi_\infty^2(r)=
  -p(p-1)L^{p-2}r^{-\frac{2(p-2)}{p-1}} \psi_\infty^2(r)
\end{equation}
($\psi_\infty$ is as in \eqref{eq:39}).
We can also take the limit in the variation of constants formula for
$ z_2(\cdot,\al)$, namely, formula \eqref{eq:42n} with $k=2$,
$f_k=f_k(\cdot,\al)$, $\psi=\psi(\cdot,\al)$, and
$\varphi=\varphi(\codt,\al)$ -- the solution of the linear equation  
\eqref{eq:28a} with the initial conditions \eqref{eq:32}. This gives
\begin{equation}
      \label{eq:42nan}
        z_2^\infty(r)=-\psi_\infty(r) \int_r^{r_0}\omega(s)f_2^\infty(s)\varphi_\infty(s)\,ds  -
  \varphi_\infty(r)\int_0^{r}\omega(s)f_2^\infty(s)\psi_\infty(s)\,ds,
\end{equation}
where $\varphi_\infty$ is the solution of \eqref{eq:170} with
$\varphi_\infty(r_0)=0$, $\varphi'_\infty(r_0)=-1/\omega(r_0)$.
Note that $\psi_\infty$, $\varphi_\infty$ are linearly independent
solutions of the homogeneous equation \eqref{eq:170} and
\eqref{eq:42nan} is a version of the variation
of constant formula for the solution $z_2^\infty$; it is valid for all
$r>0$.

We next use the relations $F'(\zeta_0)=G'(\zeta_0)$ and \eqref{eq:50}
to show that the function $r^{2/(p-1)}z_2^\infty(r)$ is bounded as
$r\to\infty$. By \eqref{eq:53}, \eqref{eq:54},
\begin{equation}
  \label{eq:55}
  G''(\zeta_0)=(\tilde\partial_\ell)^2u_{r}(r_0,\ell)\rest_{\ell=L}.  
\end{equation}
Just like  $z_2^\infty$, 
the function
\begin{equation}
  \label{eq:51}
  \tilde z_2(r):=(\tilde\partial_\ell)^2u(r,\ell)\rest_{\ell=L}
= \frac{1}{u_\ell(r_0,L)}\partial_\ell\left(\frac{ u_\ell(r,\ell)}{u_\ell(r_0,\ell)}\right)\rest_{\ell=L}
\end{equation}
is a solution a nonhomogeneous linear equation, namely,
equation \eqref{eq:29aell} with the function $u_\ell(r,L)$ on the
right-hand side replaced by the function
\begin{equation*}
  \tilde\partial_\ell u(r,\ell)\rest_{\ell=L}=\frac{ u_\ell(r,L)}{u_\ell(r_0,L)}.
\end{equation*}
From Step 2 we know that this function is identical to
$\psi_\infty$. Thus $z_2^\infty$ and $\tilde z_2$ both solve equation 
\eqref{eq:29ann}. Also, since   $\tilde\partial_\ell u(r_0,\ell)=1$
for all $\ell\approx L$, we have 
$\tilde z_2(r_0)=0$. Thus, \eqref{eq:55} and \eqref{eq:29bnn} 
imply that
$\tilde z_2\equiv z_2^\infty$. The boundedness of
the function $r^{2/(p-1)}z^\infty_2(r)$ as $r\to\infty$
is now a consequence of \eqref{eq:51}, Remark \ref{rm:higher}
and the fact that the function $r^{2/(p-1)}u_\ell(r,L)$ 
is bounded as $r\to\infty$ (cp. Lemma \ref{le:shoot2}).

We now prove \eqref{eq:45}, making use of \eqref{eq:42nan}.
(Alternatively, one could invoke the Fredholm alternative for the
nonhomogeneous equation \eqref{eq:29ann} after estimating the solution
$z_2^\infty$ and its derivative near $r=0$.)
We need some information on the asymptotics
of the function $\varphi_\infty(r)$ as $r\to\infty$. 
Recall that the asymptotics of 
$\psi_\infty$ is the same as the asymptotics of the function $\psi_1$
given in \eqref{eq:31}. 

Similarly as for the functions  $\psi$,
$\varphi$, the Wronskian of the functions  $\varphi_\infty$,
$\psi_\infty$ satisfies the following identity
\begin{equation*}
 \psi'_\infty(r)\varphi_\infty(r)-\psi_\infty(r)\varphi'_\infty(r)=\frac{1}{\omega(r)} \quad(r>0). 
\end{equation*}
Therefore, for large enough $r$ we have ($\psi_\infty(r)\ne 0$ and)
\begin{equation}
  \label{eq:56}
  \frac{d}{dr}\,\frac{\varphi_\infty(r)}{\psi_\infty(r)}=-\frac1{\omega(r)\psi^2_\infty(r)}.
\end{equation}
Hence, for any $R>0$ there is a constant $c$ such that
\begin{equation*}
  \varphi_\infty(r)=\psi_\infty(r)\left(c-\int_R^r
  \frac1{\omega(s)\psi^2_\infty(s)}\,ds\right).  
\end{equation*}
Using this, the asymptotics of $\psi_\infty$, and expression
\eqref{eq:52},  one shows via a simple computation that the function
$\omega f_2^\infty\varphi_\infty$ is integrable on $(r_0,\infty)$.
Also, the growth of $\varphi_\infty$ and the boundedness of the function 
$r^{2/(p-1)}z_2^\infty(r)$ as $r\to\infty$
imply that the coefficient of $\varphi_\infty(r)$ 
in \eqref{eq:42nan} approaches 0 as $r\to\infty$, that
is, \eqref{eq:45} holds.   This proves our claim. 

It remains to prove that \eqref{eq:45} does not hold.
Recall, that we assuming that
 $\la=0$ is an eigenvalue  of \eqref{eq:17}, or, in other words, 
that for some $j\ge 2$ we have 
\begin{equation}
    \label{eq:18nn}
   \la_j= \frac{\be}2+\frac{1}{p-1}+j=0
\end{equation}
(cp. Lemma \ref{le:HV}). Also recall that $\psi_\infty$
is an eigenfunction associated  with $\la_j$. 
As shown in \cite{Herrero-V1,Mizoguchi:blowup}, up to a scalar
multiple, 
$\psi_\infty(r)=r^\beta M_j(r^2/4)$, 
where
\begin{equation}
  \label{eq:59}
  M_j(z):= M\Bigl(-j,\beta+\frac N2,z\Bigr)
\end{equation}
is the standard Kummer function.
The assumption $\la_j=0$ yields formula  \eqref{eq:46} for $p=p_j$
and $b:=\beta+\frac N2$ can be  expressed as 
\begin{equation}
  \label{eq:58}
  b= (N-4(j-1)^2)/(4j-2).
\end{equation}
Note that \eqref{eq:43nnot} 
implies that  $b-j-1>0$.

Using \eqref{eq:59}, \eqref{eq:58}
in the integral in \eqref{eq:45} and making the
substitution $z=s^2/4$, we see that \eqref{eq:45} is equivalent to 
the following relation
\begin{equation}
  \label{eq:57a}
   \int_0^{\infty}z^{b-j-2}e^{-z}M_j^3(z)\,dz=0.
\end{equation}
Since $j$ is an integer, the Kummer function $M_j$ can
be expressed in terms of a generalized Laguerre polynomial of degree $j$:
\begin{equation*}
  M(-j, b, z) = \frac{\Ga(j+1)\Ga(b)}
  {\Ga(j+b)}L(j, b-1, z)
\end{equation*}
($\Ga$ stands for the standard Gamma function).
Thus  we have the following relation equivalent 
to \eqref{eq:45}: 
\begin{equation}
  \label{eq:57b}
   \int_0^{\infty}z^{b-j-2}e^{-z}L^3(j, b-1, z)\,dz=0.
\end{equation}
Since $B:=b-j-1$ is positive, Proposition \ref{prop-Lag}
in the appendix shows that
the integral in \eqref{eq:57b} is positive. 
Thus \eqref{eq:45} does not hold, and the proof of Proposition
\ref{prop:derivatives} is complete.

\begin{remark}
  \label{rm:last}{\rm
  Returning to Remark \ref{rm:ketc},
  we comment on the validity of  relations \eqref{eq:23}, which could
  be used instead of Step 5 and the appendix in the
  proof of \eqref{eq:20}. By \eqref{eq:30} and the chain rule,
  the function 
  $s^{N-1}f_k(s,\al)\psi(s,\al)$ in \eqref{eq:28}
  contains in particular 
  the term
  $$p(p-1)\dots(p-k+1)s^{N-1}(w(s,\al))^{p-k}\psi^{k+1}(s,\al)$$ 
  whose
  limit as $\al\upto\infty$ is 
  $$L^{p-k} p(p-1)\dots(p-k+1)s^{N-1-\frac{2(p-k)}{p-1}}  (\psi_\infty(s))^{k+1}$$ 
  (cp. \eqref{eq:29}). This term has the singularity of
  \begin{equation*}
    s^ {N-1-\frac{2(p-k)}{(p-1)}+(k+1)\be}
  \end{equation*}
  at $s=0$. Since $\be+\frac{2}{p-1}=-2j$
  for some $j\ge 2$ (cp. \eqref{eq:18nn}),
  this singularity is not integrable near $0$ if $k$ is large enough.
  This makes it reasonable to expect that \eqref{eq:23} holds for 
  some $k\geq2$. 
  However, to make this into a proof, one would need to account
  for all the other terms in $s^{N-1}f_k(s,\al)\psi(s,\al)$ obtained
  from \eqref{eq:30}. It is difficult to keep track
  of  possible  cancellations of the singularities of these terms
  in the limit as $\al\to\infty$. 
  }
\end{remark}

\section{Appendix:  Integrals with Laguerre polynomials}

 \def\comb#1#2#3{\genfrac(){0pt}{#1}{#2}{#3}}

\begin{proposition}  \label{prop-Lag}
  Let
  \begin{equation}
    \label{eq:60}
    Q_j(B):=\frac1{\Gamma(B)}\int_0^\infty x^{B-1}e^{-x}L^3(j,B+j,x)\,dx,
  \end{equation}
where $B>0$, $j\geq2$ and
\begin{equation} \label{Lag1}
L(j,\alpha,x):=\sum_{i=0}^j(-1)^i\Bigl({j+\alpha\atop j-i}\Bigr){x^i\over i!}
\end{equation}
is the generalized Laguerre polynomial.
Then $Q_j$ is a polynomial in $B$
with positive coefficients; in particular, $Q(B)>0$ for any $B>0$. 
\end{proposition}

Positivity of similar integrals involving Laguerre polynomials has
been established in a number of combinatorics papers
(see for example \cite{Foata-Z, SMA}
and references therein). However, in these papers special relations
between the exponent of $x$ and the second argument of $L$ are needed,
and we were not able to make use of the integrals or techniques in
these papers for proving Proposition \ref{prop-Lag}.
Our proof is completely independent.  

\begin{proof}[Proof of Proposition \ref{prop-Lag}]
Given integers $n_1,n_2$, set
$$T_{n_1}^{n_2}(B):=\prod_{n=n_1}^{n_2}(B+n)
\qquad \hbox{($T_{n_1}^{n_2}(B)=1$ if $n_1>n_2$)}.
$$
Let $0\leq i,m,n\leq j$. 
The recurrence relation for the Laguerre polynomials
and the orthogonality of these polynomials 
(see \cite[Chapter 22]{AS:bk}, for example) 
give 
\begin{equation} \label{Lag-recurr}
L(j,B+j,x)=\sum_{m=0}^j\Bigl({2j-i-m\atop j-m}\Bigr)L(m,B+i-1,x),
\end{equation}
and
\begin{equation} \label{Lag-ortho}
\begin{aligned}
\int_0^\infty &x^{B+i-1}e^{-x}L(m,B+i-1,x)L(n,B+i-1,x)\,dx  \\
 &= \frac{\Gamma(B+m+i)}{m!}\delta_{nm} 
  = \frac{\Gamma(B)}{m!}T_0^{m+i-1}(B)\delta_{nm}.
\end{aligned}
\end{equation}
Using \eqref{Lag1}, \eqref{Lag-recurr}, and  \eqref{Lag-ortho}, 
we obtain
$$\begin{aligned} 
  Q_j(B) &=\frac1{\Gamma(B)}\sum_{i=0}^j(-1)^i\Bigl({B+2j\atop j-i}\Bigr)\frac1{i!}
  \int_0^\infty x^{B+i-1}e^{-x}L^2(j,B+j,x)\,dx \\
 &=\frac1{j!}\sum_{i=0}^j(-1)^i\Bigl({j\atop i}\Bigr)
 T_{j+i+1}^{2j}(B)
  \sum_{m=0}^j\Bigl({2j-i-m\atop j-m}\Bigr)^2\,\frac1{m!}T_0^{m+i-1}(B) , 
\end{aligned}$$
hence
$$ 
 Q_j(B) = \frac1{j!}T_0^{2j}(B) \sum_{m=0}^j\frac1{m!}S_{j,m}(B),
$$
where
$$ 
  S_{j,m}(B):=
   \sum_{i=0}^j(-1)^i\Bigl({j\atop i}\Bigr)\Bigl({2j-i-m\atop j-m}\Bigr)^2
    \frac1{T_{m+i}^{j+i}(B)}.
$$
We show that the polynomial $T_0^{2j}(B) S_{j,m}(B)$ has positive coefficients
for each $m=0,1,\dots,j$.

Given $j\geq0$, $0\leq k_1\leq k_2$ and $B>k_2-j$, set
$$ S(j,k_1,k_2,B):=\sum_{i=0}^j (-1)^i \Bigl({j\atop i}\Bigr)
  \Bigl({j+k_1-i\atop k_1}\Bigr)^2 \frac1{T_{j-k_2+i}^{j+i}(B)}.$$ 
Notice that $S(j,j-m,j-m,B)=S_{j,m}(B)$.

Using the identities
$$\Bigl({j+1\atop i}\Bigr)=\Bigl({j\atop i}\Bigr)+\Bigl({j\atop i-1}\Bigr), \quad
T_{j+1-k+i}^{j+1+i}(B)=T_{j-k+i}^{j+i}(B+1)$$
and the induction in $j$,
one easily obtains
\begin{equation} \label{Lag2}
 S(j,0,k_2,B)=\frac{(j+k_2)!}{k_2!}\,\frac1{T_{j-k_2}^{2j}(B)},
\qquad j,k_2\geq0,\ B>k_2-j. 
\end{equation}

Next assume $j\geq2$, $0<k_1\leq k_2$, and $B>k_2-j$. 
Using the identity
$$ (j+k_1-i)^2 = (j+k_1)^2-i(2j+2k_1-1)+i(i-1) $$
and denoting
$$ T_i:=T_{j-k_2+i}^{j+i}(B)=T_{j-1-k_2+i-1}^{j-1+i-1}(B+2)=T_{j-2-k_2+i-2}^{j-2+i-2}(B+4) $$
we obtain
$$ \begin{aligned}
&S(j,k_1,k_2,B)  \\ 
&=\frac1{k_1^2}
\sum_{i=0}^j (-1)^i \Bigl({j\atop i}\Bigr) 
  \Bigl({j+k_1-1-i\atop k_1-1}\Bigr)^2 
\frac{[(j+k_1)^2-i(2j+2k_1-1)+i(i-1)]}{T_i} \\ 
&= \frac{(j+k_1)^2}{k_1^2}\sum_{i=0}^j (-1)^i \Bigl({j\atop i}\Bigr)
   \Bigl({j+k_1-1-i\atop k_1-1}\Bigr)^2 
 \frac1{T_i} \\ 
 &\quad+ \frac{(2j+2k_1-1)j}{k_1^2}\sum_{i=1}^{j} (-1)^{i-1} \Bigl({j-1\atop i-1}\Bigr)
     \Bigl({j-1+k_1-1-(i-1)\atop k_1-1}\Bigr)^2 
  \frac1{T_i} \\ 
 &\quad+ \frac{j(j-1)}{k_1^2}\sum_{i=2}^{j} (-1)^{i-2} \Bigl({j-2\atop i-2}\Bigr)
    \Bigl({j-2+k_1-1-(i-2)\atop k_1-1}\Bigr)^2 
 \frac1{T_i} \\ 
&=\frac{(j+k_1)^2}{k_1^2}S(j,k_1-1,k_2,B)
 + \frac{(2j+2k_1-1)j}{k_1^2}S(j-1,k_1-1,k_2,B+2) \\
 &\qquad\qquad+ \frac{j(j-1)}{k_1^2}S(j-2,k_1-1,k_2,B+4).
\end{aligned}$$
Repeating this argument finitely many times, we obtain
\begin{equation} \label{Lag3}
 S(j,k_1,k_2,B)=\sum_{(\tilde j,\tilde k_1)\in A} 
  c_{\tilde j,\tilde k_1}S(\tilde j,\tilde k_1,k_2,B+2(j-\tilde j)),
\end{equation}
where 
$$A=\{(\tilde j,\tilde k_1): 0\leq\tilde j\leq j,\ 0\leq\tilde k_1\leq k, \hbox{  either }
   \tilde k_1=0\hbox{ or }\tilde j\leq1\},$$
$c_{\tilde j,\tilde k_1}\geq0$ and $\sum_A c_{\tilde j,\tilde k_1}>0$.
 
Fix $j\geq2$, $0\leq m\leq j$, $k_1=k_2=j-m$, $B>0$, and
let $(\tilde j,\tilde k_1)\in A$.
Then \eqref{Lag2} and the definition of $S(j,k_1,k_2,B)$ imply
$$  S(\tilde j,\tilde k_1,k_2,B+2(j-\tilde j))
 = \begin{cases}
 \displaystyle \frac{(\tilde j+k_2)!}{k_2!}\,\frac1{T_{j-k_2+(j-\tilde j)}^{2j}(B)} 
   &\hbox{if $\tilde k_1=0$}, \\ \\
  \displaystyle \frac1{T_{2j-k_2}^{2j}(B)} 
   &\hbox{if $\tilde j=0$}, \\ \\
  \displaystyle \frac{(\tilde k_1^2+2\tilde k_1)(B+2j)+k_2+1}{T_{2j-1-k_2}^{2j}(B)}
    &\hbox{if $\tilde j=1$}.
\end{cases}$$
This and \eqref{Lag3} imply the desired conclusion. 
\end{proof}

\bibliographystyle{amsplain} 

\providecommand{\bysame}{\leavevmode\hbox to3em{\hrulefill}\thinspace}
\providecommand{\MR}{\relax\ifhmode\unskip\space\fi MR }
\providecommand{\MRhref}[2]{%
  \href{http://www.ams.org/mathscinet-getitem?mr=#1}{#2}
}
\providecommand{\href}[2]{#2}

\end{document}